\documentclass[letterpaper, 11pt,  reqno]{amsart}

\usepackage{amsmath,amssymb,amscd,amsthm,amsxtra, esint}

\usepackage{hyperref}

\allowdisplaybreaks[2]

\newtheorem{theorem}{Theorem} [section]

\newtheorem{lemma}[theorem]{Lemma}
\newtheorem{proposition}[theorem]{Proposition}
\newtheorem{remark}[theorem]{Remark}

\newtheorem{definition}[theorem]{Definition}
\newtheorem{corollary}[theorem]{Corollary}

\DeclareMathOperator*{\supp}{supp}

\newcommand{\noi}{\noindent}
\newcommand{\Z}{\mathbb{Z}}
\newcommand{\R}{\mathbb{R}}
\newcommand{\C}{\mathbb{C}}
\newcommand{\T}{\mathbb{T}}

\let\Re=\undefined\DeclareMathOperator*{\Re}{Re}
\let\Im=\undefined\DeclareMathOperator*{\Im}{Im}

\let\P= \undefined
\newcommand{\P}{\mathbf{P}}

\newcommand{\E}{\mathbb{E}}

\renewcommand{\L}{\mathcal{L}}

\newcommand{\F}{\mathcal{F}}

\newcommand{\al}{\alpha}

\newcommand{\dl}{\delta}

\newcommand{\Dl}{\Delta}
\newcommand{\eps}{\varepsilon}

\newcommand{\g}{\gamma}
\newcommand{\G}{\Gamma}

\newcommand{\s}{\sigma}
\newcommand{\Si}{\Sigma}
\newcommand{\ft}{\widehat}

\newcommand{\wt}{\widetilde}
\newcommand{\cj}{\overline}
\newcommand{\dx}{\partial_x}
\newcommand{\dt}{\partial_t}

\renewcommand{\l}{\ell}
\renewcommand{\o}{\omega}
\renewcommand{\O}{\Omega}

\newcommand{\les}{\lesssim}
\newcommand{\ges}{\gtrsim}

\newcommand{\jb}[1]
{\langle #1 \rangle}

\newcommand{\ind}{\mathbf 1}

\newcommand{\N}{\mathbb{N}}

\newtheorem*{ackno}{Acknowledgements}

\newcommand{\Pk}{\mu_k}

\newcommand{\Pkn}{\mu_{k, N}}

\numberwithin{equation}{section}
\numberwithin{theorem}{section}

\begin{document}
\baselineskip = 14pt

\title[On invariant Gibbs measures for gKdV]
{On invariant Gibbs measures for the generalized KdV equations}

\author[T.~Oh, G.~Richards, L.~Thomann]
{Tadahiro Oh, Geordie Richards, and Laurent Thomann}

\address{
Tadahiro Oh, School of Mathematics\\
The University of Edinburgh\\
and The Maxwell Institute for the Mathematical Sciences\\
James Clerk Maxwell Building\\
The King's Buildings\\
Peter Guthrie Tait Road\\
Edinburgh\\ 
EH9 3FD\\
 United Kingdom}

\email{hiro.oh@ed.ac.uk}

\address{
Geordie Richards\\
  Department of Mathematics\\
University of Rochester\\
915 Hylan Building\\
RC Box 270138, Rochester, NY 14627\\
USA} 

\email{g.richards@rochester.edu}

\address{
Laurent Thomann\\
Institut  \'Elie Cartan, Universit\'e de Lorraine, B.P. 70239,
F-54506 Vand\oe uvre-l\`es-Nancy Cedex, France}

\email{laurent.thomann@univ-lorraine.fr}



\begin{abstract}

We consider the 
generalized KdV equations on the circle.
In particular, we construct global-in-time solutions
with initial data distributed according to the Gibbs measure
and show that the law of the random solutions, at any time,
is again given by the Gibbs measure.
In handling a nonlinearity of an arbitrary high degree, 
we make use of the Hermite polynomials and
the white noise functional.

\end{abstract}

%
\maketitle
%

\baselineskip = 14pt

\section{Introduction}

\subsection{Generalized KdV equations}

We consider 
the generalized KdV equation (gKdV) on the circle:
\begin{align}
\begin{cases}
\dt u + \dx^3 u = \pm 
\dx (u^k)\\
u|_{t = 0} = u_0,
\end{cases}
\qquad ( t, x) \in \R \times \T,
\label{gKdV1}
\end{align}

\noi
where $k \geq 2$ is an integer 
and $u$ is a real-valued function 
on $\R\times \T$ with $\T = \R/(2\pi \Z)$.
When $k = 2$ and $3$, 
the equation\;\eqref{gKdV1} 
corresponds to
the famous Korteweg-de Vries equation (KdV)
and the modified KdV equation (mKdV), respectively,
and has been studied extensively from both
theoretical and applied points of view.

The gKdV equation\;\eqref{gKdV1} is known to possess 
the following Hamiltonian structure:
\[ \dt u = \dx \frac{d \mathcal{E}(u)}{du}\]

\noi
where $\mathcal{E}(u)$ is  the Hamiltonian given by 
\begin{align*}
\mathcal{E}(u) = \frac{1}{2}\int_\T (\dx u)^2dx
\pm \frac1{k+1} \int_\T u^{k+1} dx.
\end{align*}

\noi
In particular, $\mathcal{E}(u)$ is conserved under the dynamics of\;\eqref{gKdV1}.
Moreover, 
the spatial mean $\int_\T u \, dx$
and the mass $M(u) = \int_\T u^2 \, dx$
are also conserved.
While 
KdV ($k = 2$) and mKdV ($k = 3$) are
 known to be completely 
integrable and thus possess
infinitely many conservation laws, 
there are no other known conservation laws for higher values of $k \geq 4$.

In view of the conservation of the spatial mean,  
we assume that both the initial condition\;$u_0$ 
and the solution\;$u$ are of spatial mean 0 in the following.
In other words, 
defining the Fourier coefficient $\ft f(n)$
by 
\[ \ft f(n) = \F(f)(n) = \int_\T f(x) e^{- inx} dx,\]
we will work on real-valued functions of the form:\footnote{Hereafter, we drop
the harmless $2\pi$.}
\[f(x) = \frac{1}{2\pi}\sum_{n \in \Z^*} \ft f(n) e^{inx}\]

\noi
where 
$\Z^* = \Z \setminus \{0\}$.

\subsection{Gibbs measures}

Consider  the following Hamiltonian dynamics on $\R^{2n}$:
\begin{equation} \label{HR2}
\textstyle \dot{p}_j = \frac{\partial \mathcal{E}}{\partial q_j} 
\qquad
\text{and}
\qquad 
\dot{q}_j = - \frac{\partial \mathcal{E}}{\partial p_j}
\end{equation}

\noi
 with Hamiltonian $ \mathcal{E} (p, q)= \mathcal{E}(p_1, \cdots, p_n, q_1, \cdots, q_n)$.
Liouville's theorem states that the Lebesgue measure
$\prod_{j = 1}^n dp_j dq_j$
on $\mathbb{R}^{2n}$ is invariant under the dynamics.
Then, it follows from the conservation of the Hamiltonian $\mathcal{E}$
that  the Gibbs measure
$e^{- \mathcal{E}(p, q)} \prod_{j = 1}^{n} dp_j dq_j$ is invariant under the dynamics of~\eqref{HR2}.

In view of the Hamiltonian structure of gKdV\;\eqref{gKdV1}, 
we expect 
 the Gibbs measure of the form:\footnote{In the following, 
$Z$,  $Z_N$, and etc.~denote various  normalizing constants
so that the corresponding measures are probability measures when appropriate.}
\begin{align}
d\Pk = Z^{-1} \exp(-  \mathcal{E}(u))du
 = Z^{-1} e^{\mp\frac 1{k+1} \int_\T u^{k+1} dx} e^{-\frac 12 \int_\T (\dx  u )^2 dx} du
\label{G1}
 \end{align}

\noi
to be invariant under the dynamics of~\eqref{gKdV1}.
As it stands,\;\eqref{G1} is merely a formal expression.
It turns out, however, that 
 the Gibbs measure $\Pk$ in\;\eqref{G1} can be defined 
as a  probability measure absolutely continuous with respect to the 
following 
Gaussian measure:
\begin{align}
 d\rho =  Z_0^{-1} e^{-\frac 12 \int_\T (\dx u)^2 dx } du.
\label{G3}
\end{align}

\noi
Note that 
 $\rho$ in~\eqref{G3}
 is  the induced probability measure
under the map:
\begin{align}
 \o \in \O \longmapsto u(x) = u(x; \o) = \sum_{n \in \Z^*} \frac{g_n(\o)}{|n|}e^{in x},
\label{G5}
 \end{align}

\noi
where $\{g_n\}_{n \in \N}$
is a sequence of independent standard\footnote{Namely, $g_n$ has mean 0
and variance 1.} 
complex-valued Gaussian
random variables on a probability space $(\O, \F, P)$
conditioned  that $g_{-n} = \cj {g_n}$, $n \in \N$.
The random function $u$ under $\rho$,  
represented by the Fourier-Wiener series in~\eqref{G5}, 
corresponds to 
  a mean-zero Brownian loop (periodic Wiener process) on $\T$.
 See\;\cite{Benyi}. 
Therefore, we refer to $\rho$
as the (mean-zero periodic) Wiener measure in the following.
Lastly, note that 
$u$ in~\eqref{G5}
lies in $H^{s}(\T)\setminus H^\frac{1}{2}(\T)$ for any $s < \frac 12 $, almost surely.

From Sobolev's inequality, we see 
that $\int_\T(u(x; \o))^{k+1} dx$ is finite almost surely.
Hence, 
in the defocusing case, i.e.~with the $+$ sign in\;\eqref{gKdV1}
and an odd integer $k\geq 3$, 
the Gibbs measure $\mu_k$
is a well-defined probability measure 
on $H^s(\T)$, $s < \frac{1}{2}$,
absolutely continuous with respect to $\rho$.

Next, let us discuss the  non-defocusing case,
i.e.~either
$k$ is even or we have  the $-$ sign in \eqref{gKdV1}, 
(corresponding to the $+$ sign in \eqref{G1}).
In this case, 
 one encounters a normalization issue in\;\eqref{G1} since $\int_{\T} u^{k+1} dx$ is unbounded.
In particular, the weight 
$e^{\mp\frac 1{k+1} \int_\T u^{k+1} dx}$
is not integrable with respect to the Gaussian measure $\rho$ in \eqref{G3}.
In\;\cite{LRS}, 
Lebowitz-Rose-Speer proposed to 
insert a mass cutoff $\ind_{\{\int u^2 dx\leq R\}}$
and consider the Gibbs measure $\mu_k$ of the following form:
\begin{align}
d\Pk 
& = Z^{-1}\ind_{\{\int u^2 dx\leq R\}} \exp(-  \mathcal{E}(u))du \notag\\
&  = Z^{-1}\ind_{\{\int u^2 dx\leq R\}} e^{\mp\frac 1{k+1} \int_\T u^{k+1} dx} e^{-\frac 12 \int_\T (\dx  u )^2 dx} du.
\label{G1a}
 \end{align}

\noi
They showed that one can realize the Gibbs measure $\mu_k$ in\;\eqref{G1a} 
as a  probability measure 
on $H^s(\T)$, $s < \frac{1}{2}$, 

\begin{itemize}
\item for any mass cutoff size $R>0$ when $1< k \leq 5$, and
\item for sufficiently small $R>0$ is  when $k = 5$.

\end{itemize}

\noi
Moreover, it was shown in \cite{LRS} that 
the Gibbs measure $\mu_k$ is non-normalizable 
when $k > 5$
or for $R \gg 1$ when $k = 5$.
Lastly, we point out that the critical value $k = 5$
corresponds the smallest power of the nonlinearity 
where 
\eqref{gKdV1} on the real line possesses
finite time blowup solutions
 \cite{MM, M}.


In the seminal work\;\cite{BO94}, 
Bourgain proved the invariance of the Gibbs measures
for KdV $(k = 2)$ and mKdV ($k = 3$).
Here, by invariance, we mean that 
\begin{align}
\mu_k \big(\Phi(-t) A\big) = \mu_k (A)
\label{inv1}
\end{align}

\noi
for any measurable set $A \in \mathcal{B}_{H^s(\T)}$, $s < \frac 12$, 
and any $t \in \R$, 
where $\Phi(t): u_0 \in H^s(\T) \mapsto 
u(t) = \Phi(t) u_0 \in H^s(\T)$
is a well-defined solution map to\;\eqref{gKdV1}, at least almost surely with respect to 
$\mu_k$.
While KdV was known to be globally well-posed
in $L^2(\T)\supset \supp \mu_2$, as shown in\;\cite{BO93},  
there was no well-posedness result for mKdV
in the support of the Gibbs measure.
In\;\cite{BO94}, Bourgain first established 
local well-posedness of mKdV in 
$H^s(\T) \cap \F L^{s_1, \infty}(\T) \supset \supp \mu_3$
for some $s < \frac 12 < s_1 < 1$,
where $\| f\|_{\F L^{s_1, \infty}(\T)} = \| \jb{n}^{s_1} \ft f(n)\|_{\l^\infty_n}$.
He then used a probabilistic argument
to construct almost sure global-in-time dynamics.
In fact, 
the  main novelty of the paper\;\cite{BO94}
is this globalization argument, exploiting
the invariance of the finite dimensional Gibbs measures
for the finite dimensional approximations to a given PDE.
There have been many results 
on the construction of invariant Gibbs measures
for Hamiltonian PDEs
that 
followed and further developed the approach in\;\cite{BO94}.
See for example\;\cite{BO96, BO97,  
TZ1, TZ2, BTIMRN, BT2, OH3,  OHSBO,  Tzv, TTz, NORS,
 SuzzoniBBM, BTT, BB1, DengBO, R}. 
 We also refer to the book by Zhidkov\;\cite[Chapter 4]{Zhid2} for the construction of infinitely many invariant measures for KdV associated to the conservation laws of the equation at different levels of Sobolev regularity.

In\;\cite{R}, the second author 
considered the same problem for the quartic gKdV ($k = 4$).
As in the case of mKdV, the main challenge in\;\cite{R} was the construction of the local-in-time dynamics
in the support of the Gibbs measure.
Following the approach in\;\cite{OH6}, 
the second author constructed almost sure local-in-time dynamics 
by establishing a probabilistic a priori estimate
and proved the invariance of the Gibbs measure.\footnote{Strictly speaking, 
the invariance of the Gibbs measure in\;\cite{R}
was shown only for the gauged quartic gKdV.
See Remark\;\ref{REM:quartic} below.}

\subsection{Main result}

We now state our main result.
In particular, 
for large values of $k \geq 5$, 
this theorem addresses 
the invariance 
of the Gibbs measures for gKdV\;\eqref{gKdV1} 
for the first time.

\begin{theorem}\label{THM:1}

Assume one of the following conditions:
\begin{itemize}
\item[\textup{(i)}] 
defocusing gKdV:  with the $+$ sign in\;\eqref{gKdV1} and odd $k \geq 3$, or

\vspace{1mm}
\item[\textup{(ii)}] 
non-defocusing gKdV:  $2 \leq k \leq 5$.
When $k = 5$, the mass threshold $R > 0$ is sufficiently small.

\end{itemize}

\noi
Then, 
given any $s < \frac 12$,
there exists a set $\Si = \Si(s)$ of full measure
with respect to $\Pk$
such that for every $\phi \in \Si$, 
the  generalized KdV equation\;\eqref{gKdV1} with mean-zero initial condition $u(0) = \phi$
has a global-in-time solution 
$ u \in C(\R; H^s(\T))$.
Moreover, for all $t \in \R$, 
the law of the random function $u(t)$ is given by~$\Pk$.

\end{theorem}

Theorem~\ref{THM:1} asserts two statements:
global 
existence of solutions (without uniqueness) and invariance of $\Pk$.
Regarding the invariance part, 
Theorem~\ref{THM:1} only claims that,  
given any $t \in \R$, the law $\L(u(t))$ of the $H^s$-valued random variable $u(t)$
is given by the Gibbs measure $\Pk$.
This implies the invariance property of the Gibbs measure $\Pk$
in some mild sense, but 
it is weaker than the actual invariance described in\;\eqref{inv1}.
While the  well-posedness results for gKdV in low regularity setting \cite{ST, CKSTT2, R} are obtained via gauge transforms, 
 we work directly on the equation\;\eqref{gKdV1} in the following.
This is crucial to study  the invariance property of the Gibbs measure $\mu_k$. 
See Remark\;\ref{REM:quartic} for more details.

\medskip

A precursor to the existence part of Theorem~\ref{THM:1}
appears in the work by the third author with Burq and Tzvetkov\;\cite{BTT0},
where they 
used the energy conservation and a regularization 
property under randomization
to construct global-in-time solutions
to the cubic NLW on\;$\T^d$ for $d \geq 3$.
The main ingredient in~\cite{BTT0} is the compactness of
the solutions to the approximating PDEs.

We prove Theorem\;\ref{THM:1}
by following the approach
presented
in the work
 by the third author with Burq and Tzvetkov~\cite{BTT1}, 
which  was in turn  motivated
by the works of Albeverio-Cruzeiro~\cite{AC}
and Da Prato-Debussche~\cite{DPD}
in the study of fluids. This method allows us to construct global dynamics, even for very rough initial conditions\;\cite{AC,GHT2}. The main idea is to exploit the invariance
of the truncated Gibbs measures $\Pkn$ (see\;\eqref{TG1} below)
and construct a tight (= compact) sequence of measures $\nu_N$ on space-time functions.  
We then apply 
Skorokhod's theorem (see Lemma~\ref{LEM:Sk} below)
to construct 
global-in-time {\it weak} solutions for 
gKdV\;\eqref{gKdV1}.

In order to prove Theorem\;\ref{THM:1} following the approach in \cite{BTT1},
 one needs a uniform bound on the nonlinearity of gKdV and its truncated version (see \eqref{gKdV3} below). 
 This is quite easy, since we have ${u\in L^p}$ 
  for all $2\leq p<\infty$, 
 almost surely with respect to the Gibbs measure $\mu_k$. 
In the following, we prove a stronger regularity result on the nonlinearity.
Recalling that $u$ defined in \eqref{G5} 
lies in $H^s(\T)$
for any $s < \frac 12$ almost surely, 
we show that 
  $u^k$ also lies in $H^s(\T)$ for any $s < \frac 12$, almost surely.
See  Proposition\;\ref{PROP:main}.
We wanted to include this optimal bound in this paper, since we believe that this could be a first step towards a probabilistic strong well-posedness result for gKdV on the support of the Gibbs measure.  The main source of difficulty in proving Proposition\;\ref{PROP:main} comes from  the more  complicated combinatorics\footnote{See \cite{BTT1}
for examples of probabilistic estimates on various nonlinearities of low degrees, 
where the required combinatorics is relatively simple thanks to the low degree of the nonlinearities.
See also \cite[Appendix A]{OT1} for a concrete combinatorial computation 
for the (Wick ordered) quintic nonlinearity.} for larger values of $k$.
In order to overcome this combinatorial difficulty,
we make use of  the {\it white noise functional} (see Definition~\ref{DEF:W} below)
and the orthogonality property of the Hermite polynomials (Lemma\;\ref{LEM:W1})
and entirely avoid 
combinatorial arguments
of increasing complexity in~$k$.
This allows us to 
prove Proposition\;\ref{PROP:main}
in a concise and uniform manner.
See  Da Prato-Debussche \cite{DPD2} and Da Prato-Tubaro\;\cite{DPT1} 
for a presentation of this method in 
the context of  the stochastic quantization equation on\;$\T^2$.
See also  a  related recent work  by the first and third authors\;\cite{OT1}
on the invariant Gibbs measures
for the nonlinear Schr\"odinger equations on\;$\T^2$.

\medskip

For simplicity of the presentation, we only treat\;\eqref{gKdV1}
with the $+$ sign
and  drop the mass cutoff $\displaystyle\ind_{\{\int u^2 dx\leq R\}}$
required for normalization of the Gibbs measures
in the non-defocusing case.
Note that 
the restriction on the values of $k$ in the non-defocusing case in Theorem\;\ref{THM:1}
simply comes from the normalization of the Gibbs measures
as probability measures\;\cite{LRS, BO94}
and that it does not appear in  the proof of Theorem\;\ref{THM:1}
in an explicit manner.

\begin{remark} \rm
The mean-zero assumption 
in the construction of the Gibbs measure
(and hence in 
Theorem\;\ref{THM:1}) is not essential.
In view of the mass conservation, 
one can consider the Gibbs measure of the form:
\begin{align}
d\ft \mu_k = Z^{-1} \exp\big(-\mathcal{E}(u) - \tfrac 12 M(u)\big)du
 = Z^{-1} e^{\mp\frac 1{k+1} \int_\T u^{k+1} dx} 
 d  \ft \rho,  
\label{D1}
 \end{align}

\noi
where $\ft \rho$ denotes the Gaussian measure given by 
\begin{align*}
 d\ft \rho =  Z_0^{-1} e^{-\frac 12 \int_\T (\dx u)^2 dx  - \frac 12 \int_\T u^2 dx } du.
\end{align*}

\noi
Under $\ft \rho$, 
a typical element $u$ is represented by 
\begin{align}
u(x) = u(x; \o) = \sum_{n \in \Z^*} \frac{g_n(\o)}{\sqrt{1+n^2}}e^{in x},
\label{D2}
 \end{align}

\noi
where $\{g_n\}_{n \in \Z_{\geq 0}}$
is a sequence of independent standard complex-valued Gaussian
random variables conditioned that $g_{-n} = \cj {g_n}$, $n \in \Z_{\geq 0}$.
Unlike\;\eqref{G5}, the random function $u$ in\;\eqref{D2} 
has a spatial mean $g_0(\o)$.
We point out that 
 Theorem\;\ref{THM:1} with  exactly the same proof 
holds for the Gibbs measure $\ft \mu_k$ in\;\eqref{D1}.

\end{remark}

\begin{remark}\label{REM:quartic} \rm

(i) Let $k \geq 4$.
On the one hand,  gKdV\;\eqref{gKdV1} is known to be locally well-posed
 in $H^s(\T)$, $s \geq \frac 12$; see\;\cite{ST, CKSTT2}.
 On the other hand, 
  it is known to be mildly ill-posed
for $s < \frac 12$
 in the sense that the solution
 map fails to be smooth in this range of regularity.
See\;\cite{BO97no2, CKSTT2}.
Therefore, it is non-trivial to construct (local-in-time) solutions
in the support of the Gibbs measure.

When $k = 4$, local-in-time dynamics was constructed in a probabilistic manner
\cite{R}.
As in the deterministic case\;\cite{ST, CKSTT2}, 
this probabilistic construction of solutions was carried out
through a gauge transform and thus the uniqueness statement
 in\;\cite{R} was very mild.
While one can apply a similar probabilistic construction of solutions 
for $k \geq 5$, 
such construction requires case-by-case consideration (see\;\cite{BO96, CO, R})
and thus 
combinatorics gets out of control for large values of $k$.
At this point, there seems to be no uniform way to 
perform this probabilistic construction for all values of $k$, 
rather than working out case-by-case analysis for each fixed value of $k$.

\medskip

\noi
(ii)  For the quartic gKdV ($k = 4$), 
the second author\;\cite{R} proved 
the invariance of the Gibbs measure $\mu_4$ 
for  the following {\it gauged} gKdV:
\begin{align*}
\dt u + \dx^3 u = \pm 
4 \big(\P_{\ne 0} [u^{3}] \big)\dx u, 
\end{align*}

\noi
where $\P_{\ne 0}$ denotes the orthogonal projection
onto the mean-zero functions.
Through the inverse gauge transform, 
this result yields almost sure global well-posedness
of the ungauged gKdV\;\eqref{gKdV1}.
The invariance of $\mu_4$ under the dynamics of\;\eqref{gKdV1}, however, 
is unknown.
Moreover, we do not know if the Gibbs measure $\mu_k$
is absolutely continuous with respect to 
 the pushforward of $\mu_k$ under the inverse gauge transform. 
This is a sharp contrast to the situation for the derivative cubic nonlinear Schr\"odinger equation.
See \cite{NORS, NRSS}.

While Theorem\;\ref{THM:1} 
asserts the existence of global-in-time dynamics
under which $\mu_4$ is invariant, 
we do not know if these solutions coincide
with the almost sure global solutions constructed in\;\cite{R}
due to the mild uniqueness statement under the inverse gauge transform.

\end{remark}

\begin{remark}\label{REM:unique}\rm
In Theorem \ref{THM:1},  
we first fix  $s < \frac 12$
and construct a (non-unique) global solution $u \in C(\R; H^s(\T))$.
Thus, the solution $u$ may depend on the choice of $s < \frac 12$.
In fact, one can easily modify the argument
and construct a global solution $u \in \bigcap_{s< \frac 12} C(\R;H^s(\T)) $,
thus removing the dependence on a specific choice of $s< \frac 12$.
See Remark \ref{REM:unique2} below.

\end{remark}

\section{On the truncated gKdV equation} \label{SEC:2}

In this section, we introduce the truncated gKdV equation
and the truncated Gibbs measure $\mu_{k, N}$
and discuss their basic properties.
Given $N \in \N$, 
define $E_N$ and $E_N^\perp$ by 
\[E_N = \text{span}\{ e^{in\cdot x}\}_{|n|\leq N}
\qquad \text{and}\qquad
E_N^\perp = \text{span}\{ e^{in\cdot x}\}_{|n|> N}.\]

\noi
Consider the following truncated gKdV on $\T$:
\begin{align}
\dt u^N + \dx^3 u^N =  \dx \P_N\big[ (\P_N u^N)^k \big], 
\label{gKdV2}
\end{align}

\noi
where $\P_N $ denotes the Dirichlet projection onto the  frequencies $\{|n| \leq N\}$.
Letting ${v^N = \P_N u^N}$, 
we can decouple\;\eqref{gKdV2} 
into the following finite dimensional system of ODEs on $E_N$: 
\begin{align}
\dt v^N + \dx^3 v^N =  \dx \P_N\big[ (v^N)^k \big]
\label{gKdV3}
\end{align}

\noi
and the linear flow for high frequencies $\{|n| > N\}$:
\begin{align}
\dt \P_N^\perp u^N + \dx^3 \P_N^\perp u^N= 0.
\label{gKdV4}
\end{align}

\noi
Here,  $\P_N^\perp $ is the Dirichlet projection onto the high frequencies $\{|n|> N\}$.
Note that the truncated gKdV\;\eqref{gKdV2}
is a Hamiltonian PDE with 
\begin{align}
\mathcal{E}_N(u^N) = \frac{1}{2}\int_\T (\dx u^N)^2dx
+ \frac1{k+1} \int_\T (\P_N u^N)^{k+1} dx.
\label{Hamil2}
\end{align}

Associated to the truncated gKdV\;\eqref{gKdV2}, 
let us define the truncated Gibbs measure $\mu_{k, N}$ by\footnote{In the non-defocusing case, 
there is a
mass cutoff $\ind_{\{\int u^2 dx\leq R\}}$ which we omit for simplicity of the presentation.}
\begin{align}
d \mu_{k, N} 
& = Z_N^{-1} \exp (- \mathcal{E}_N(u^N)) du^N 
 = Z_N^{-1} 
R_N(u) 
d \rho (u), 
\label{TG1}
\end{align}

\noi
where $\rho$ is the Wiener measure defined in\;\eqref{G3}
and $R_N(u)$ is defined by 
\begin{align*}
R_N(u) 
:=  e^{- \frac 1{k+1} \int_\T (\P_N u )^{k+1} dx} .
\end{align*}

\noi
It basically follows from the argument in 
\;\cite{LRS, BO94}
that $R_N(u)$ converges to $R_\infty(u)$
in $L^p(\rho)$, $1\leq p < \infty$, as $N \to \infty$.
Consequently, for any $1 \leq p < \infty$, 
we have
\begin{align}
\| R_N(u) \|_{L^p(\rho)} \leq C_p < \infty, 
\label{inv8}
\end{align}

\noi
uniformly in $N \in \N$, 
and
\begin{align}
\lim_{N \to \infty} \mu_{k, N}(A) = \mu_k(A)
\label{inv9}
\end{align}

\noi
for any measurable set $A \in \mathcal{B}_{H^s(\T)}$, $s < \frac 12$.
See also\;\cite{BTT1, OT1}.

We now decompose the Wiener measure $\rho$
as  
\[\rho = \rho_N \otimes \rho_N^\perp,\] 

\noi
where $\rho_N$  
and $\rho_N^\perp$ are
the marginals of $\rho$ on $E_N$ and $E_N^\perp$, respectively.
Then,  we can write the truncated Gibbs measure $\Pkn$ in~\eqref{TG1} as 
\begin{align}
\Pkn = \ft \mu_{k, N} \otimes \rho^\perp_N,
\label{NGibbs1}
\end{align} 

\noi
where 
$\ft \mu_{k, N}$ is the finite dimensional Gibbs measure defined by 
\begin{align*}
d\ft \mu_{k, N} =  Z_N^{-1} e^{-\frac 1{k+1} \int_{\T} (\P_N u^N)^{k+1}  dx} d\rho_N.
\end{align*}

We have the following lemma on global well-posedness of 
the truncated gKdV\;\eqref{gKdV2}
and the invariance of the truncated Gibbs measure $\mu_{k, N}$
under the dynamics of\;\eqref{gKdV2}.

\begin{lemma}\label{LEM:global}
Let  $N \in \N$ and $s< \frac 12$.
Then, 
the truncated gKdV~\eqref{gKdV2}
is globally well-posed in $H^s(\T)$.
Moreover, the truncated Gibbs measure $\Pkn$  
is invariant under the dynamics of\;\eqref{gKdV2}.

\end{lemma}

In particular, Lemma\;\ref{LEM:global}
states that if the law of $u^N(0)$ is given by $\Pkn$, then
 the law of the corresponding solution $u^N(t)$ is again given by $\Pkn$
for any $t \in \R$.

\begin{proof}

We first prove global well-posedness of 
the truncated gKdV~\eqref{gKdV2}.
We use the decomposition of\;\eqref{gKdV2}
by the low frequency part\;\eqref{gKdV3} and 
the high frequency part\;\eqref{gKdV4}.
As a linear equation, 
the high frequency part\;\eqref{gKdV4}
is globally well-posed.
By viewing~\eqref{gKdV3} on the Fourier side,
we see that~\eqref{gKdV3} 
is a finite dimensional system of   ODEs of dimension\;$2N$.
Hence, by the Cauchy-Lipschitz theorem, 
\eqref{gKdV3} is locally well-posed.

By a direct computation, it is easy to see from\;\eqref{gKdV3}
that $\int_\T |v^N|^2 dx$ is conserved for\;\eqref{gKdV3}.
In particular, this shows that 
the Euclidean norm  on the phase space  $\C^{2N}$
\[\big\| \{\ft{v^N}(n)\}_{|n|\leq N}\big\|_{\C^{2N}}
= \bigg(\sum_{|n| \leq N} |\ft{v^N}(n)|^2\bigg)^\frac{1}{2}
= \bigg(\int_{\T} |v^N|^2 dx \bigg)^\frac{1}{2}\]

\noi
is conserved under~\eqref{gKdV3}.
This proves global existence for~\eqref{gKdV3}
and hence for the truncated gKdV~\eqref{gKdV2}.

On the one hand,  
the linear flow\;\eqref{gKdV4} leaves 
the Gaussian measure $\rho_N^\perp$ on $E_N^\perp$
 invariant under the dynamics.
On the other hand, 
noting that~\eqref{gKdV3} is the finite dimensional Hamiltonian
dynamics corresponding to $\mathcal{E}_N(v^N)$ defined in\;\eqref{Hamil2}, 
we see that 
$\ft \mu_{k, N}$
 is invariant under~\eqref{gKdV3}.
Therefore, in view of\;\eqref{NGibbs1}, 
the truncated Gibbs measure
$\Pkn $ is invariant under
the dynamics of~\eqref{gKdV2}.
\end{proof}

\section{Hermite functions and white noise functional}\label{SEC:3}

Let $u$ be 
the random function defined in\;\eqref{G5}  distributed according to the Wiener  measure $\rho$.
Then, 
the nonlinearity $\dx (u^k)$ makes sense as a (spatial) 
distribution almost surely, 
since 
$u \in H^{\frac 12 - \eps} (\T)$ for any $\eps> 0$ 
and hence
$u \in L^p (\T)$ for any $p < \infty$ almost surely.
Given $N \in \N$, define $F_N(u)$ and $F(u)$ by 
\begin{equation}
 F_N(u) :=  \P_N [( \P_N u)^k]\quad \text{and}\quad 
F(u)  := F_\infty(u) = u^k.
\label{nonlin1}
\end{equation}

\noi
The main goal of this section is
to establish the following convergence property of $F_N(u)$ to $F(u)$.

\begin{proposition}\label{PROP:main}
Let $k \geq 2$ be an integer
and  $ s < \frac 12$.
Then,  there exists $C_{k, s} > 0$ such that 
\begin{align}
  \|F_N(u)  \|_{L^p(\rho; H^s)}, \ \| F(u) \|_{L^p(\rho; H^s)}
\leq C_{k, s}  (p-1)^\frac{k}{2} 
\label{X1}
\end{align}

\noi
for any $p \geq 1$
and 
 any $N \in \N$.	
Moreover, 
given $\eps >0 $ with $s + \eps < \frac 12$, 
there exists $C_{k, s, \eps} > 0$ such that 
\begin{align}
\| F_M(u) - F_N(u) \|_{L^p(\rho; H^s)}
\leq C_{k, s, \eps} (p-1)^\frac{k}{2}\frac{1}{N^\eps}
\label{X2}
\end{align}

\noi
for 
 any $p \geq 1$ and 
any $ 1 \leq N \leq M \leq \infty$.
In particular, $F_N(u)$ converges to $F(u)$ in $L^p (\rho; H^s(\T))$
as $N \to \infty$.

\end{proposition}

\begin{remark} \rm
In order to construct global-in-time  weak solutions claimed in Theorem \ref{THM:1}, 
one only needs to prove 
in \eqref{X1} and \eqref{X2}
with $s = 0$.
This easily follows from the fact that $u \in L^q(\T)$ for any $q < \infty$ almost surely.
Then, one can proceed as in \cite[Lemma 5.6]{BTT1}.
On the other hand, 
Proposition \ref{PROP:main}
is optimal in the range of $s < \frac 12 $
and shows the stability of $u \mapsto F_N(u)$ in the $H^s$-norm.

\end{remark}

\subsection{Hermite polynomials
and white noise functional}
First, recall from\;\cite{Kuo} the Hermite polynomials $H_n(x; \s)$
defined through the generating function:
\begin{equation}
G(t, x; \s) : =  e^{tx - \frac{1}{2}\s t^2} = \sum_{k = 0}^\infty \frac{t^k}{k!} H_k(x;\s).
\label{H1}
 \end{equation}
	
\noi
For simplicity, we set $G(t, x) : = G(t, x; 1)$
and $H_k(x) : = H_k(x; 1)$ in the following.
For readers' convenience, we write out the first few Hermite polynomials:
\begin{align*}
& H_0(x; \s) = 1, 
\qquad 
H_1(x; \s) = x, 
\qquad
H_2(x; \s) = x^2 - \s, \\
& H_3(x; \s) = x^3 - 3\s x, 
\qquad 
H_4(x; \s) = x^4 - 6\s x^2 +3\s^2.
\end{align*}
	
\noi
Then, the monomial $x^n$
can be expressed in term of the Hermite polynomials:
\begin{align}
x^k = \sum_{m = 0}^{[\frac{k}{2}]}
\begin{pmatrix}
k\\2m 
\end{pmatrix}
(2m-1)!! \, \s^m H_{k-2m}(x; \s),
\label{H2}
\end{align}

\noi
where $(2m-1)!! = (2m-1)(2m-3)\cdots 3\cdot 1
= \frac{(2m)!}{2^m m!}$
and $(-1)!!:= 1$ by convention.

Next, we define the white noise functional.
Let  $w(x;\o)$ be the (real-valued) mean-zero Gaussian white noise on $\T$
defined by
\[ w(x;\o) = \sum_{n\in \Z^*} g_n(\o) e^{inx}.\]

%
%
%
%

\begin{definition}\label{DEF:W}\rm
The {\it white noise functional} $W_{(\cdot)}: L^2(\T) \to L^2(\O)$
is defined by 
\begin{equation}
 W_f (\o) = \jb{f, w(\o)}_{L^2_x} = \sum_{n \in \Z^*} \ft f(n) \cj{g_n}(\o)
\label{W0}
 \end{equation}

\noi
for a real-valued function $f \in L^2(\T)$.
Here, 
 $\{g_n\}_{n \in \N}$
is a sequence of independent standard
 complex-valued Gaussian
random variables 
conditioned that $g_{-n} = \cj {g_n}$, $n \in \N$, as in\;\eqref{G5}.

\end{definition}

For real-valued $f \in L^2(\T)$,  
$W_f$ is a real-valued Gaussian random variable
with mean 0 and variance $\|f\|_{L^2}^2$.
Moreover, we have 
\[ E\big[ W_f W_h ] = \jb{f, h}_{L^2_x}\]

\noi
for $f, h \in L^2(\T)$.
In particular, the white noise functional
$W_{(\cdot)}$ is an isometry from $L^2(\T)$ onto $L^2(\O)$.

The following orthogonality lemma on the white noise functional
and Hermite polynomials is well known\;\cite{DPT1}
and will play an essential role in the subsequent analysis.
We  present the proof for readers' convenience.

\begin{lemma}\label{LEM:W1}
Let $f, h \in L^2(\T)$ such that $\|f\|_{L^2} = \|h\|_{L^2} = 1$.
Then, for $k, m \in \Z_{\geq 0}$, we have 
\begin{align}
\E\big[ H_k(W_f)H_m(W_h)\big]
=  \dl_{km} k! [\jb{f, h}_{L^2_x}]^k.
\label{W1}
\end{align}

\noi
Here, $\dl_{km}$ denotes the Kronecker delta function.
\end{lemma}

\begin{proof}

First recall the following identity:
\begin{align*}
\int_\O e^{W_{f}(\o)} dP
& = \prod_{n \in \N} \frac{1}{\pi}\int_{\C} e^{2\Re (\ft f(n) \cj g_n)} e^{- |g_n|^2} dg_n \notag\\
& = \prod_{n \in \N}\frac{1}{\pi} \int_{\R} e^{2\Re \ft f(n) \Re g_n} e^{- (\Re g_n)^2} d\Re g_n
\int_{\R} e^{2\Im \ft f(n) \Im g_n} e^{- (\Im g_n)^2} d\Im g_n \notag\\
& =  e^{\sum_{n \in \N} |\ft f(n)|^2} = e^{\frac 12 \|f \|_{L^2}^2}.
\end{align*}

Let $G$ be as in\;\eqref{H1}.
Then, for any $t, s \in \R$ and $f, h \in L^2(\T)$ with $\|f \|_{L^2} = \|h\|_{L^2} = 1$, we have 
\begin{align}
\int_{\O} G(t, W_f(\o)) G(s, W_h(\o)) dP(\o)
& = e^{-\frac{t^2 + s^2}{2}} 
\int_\O e^{W_{tf + sh}(\o)} dP(\o)\notag \\
& = e^{-\frac{t^2 + s^2}{2}} e^{\frac{1}{2}\|tf + sh\|_{L^2}^2}
= e^{ts \jb{f, h}_{L^2_x}}.
\label{W1a}
\end{align}

\noi
Thus, it follows from\;\eqref{H1} and\;\eqref{W1a} that 
\begin{align*}
e^{ts \jb{f, h}_{L^2}}
 = \sum_{k, m  = 0}^\infty 
 \frac{t^ks^m}{k!m!} 
\int_\O H_k(W_f(\o))H_m(W_h(\o))dP(\o).
\end{align*}

\noi
By comparing the coefficients of $t^ks^m$, we obtain\;\eqref{W1}.
\end{proof}

Given $N \in \N \cup \{\infty\}$, 
define
\begin{align*}
\s_N := \E\big[\|\P_{N} u \|_{L^2}^2 \big]= \sum_{1\leq |n| \leq N} \frac{1}{n^2}
\end{align*}

\noi
with the understanding that $\P_\infty  = \text{Id}$.
For {\it fixed} $x \in \T$ and 
 $N \in \N \cup \{\infty\}$, 
 we also define
\begin{align}
\eta_N(x) (\cdot) & := \frac{1}{\s_N^\frac{1}{2}}
\sum_{1\leq |n| \leq N} \frac{\cj{e_n(x)}}{|n|}e_n(\cdot), 
\label{W3}\\
\g_N (\cdot) & := 
\sum_{1\leq |n| \leq N} \frac{1}{n^2}e_n(\cdot),\nonumber
\end{align}
	\noi
where $e_n(y) = e^{iny}$.
Note that 
\begin{align}
 \| \eta_N(x)\|_{L^2(\T)} = 1
\label{W3b}
\end{align}	

\noi
for all  fixed $x \in \T$ and all $N \in \N\cup\{\infty\}$.
Moreover, we have 
\begin{align}
\jb{\eta_M(x), \eta_N(y)}_{L^2_x}
= \frac{1}{\s_M^\frac{1}{2}\s_N^\frac{1}{2}} \g_N(y-x), 
\label{W4}
\end{align}

\noi
for fixed $x, y\in \T$
and $N, M \in \N \cup \{\infty\}$ with $M\geq N$. 
Note that $\s_N \leq \s_\infty = \frac{\pi^2}{3}$ for all $N \in \N$.

We now establish a second moment bound on the Fourier coefficients
of the (truncated) nonlinearity $F_N(u)$ and $F(u)$ defined in\;\eqref{nonlin1}.

\begin{lemma}\label{LEM:W5}
Let $k \geq 2$ be an integer.
Then, there exists $C_k > 0$ such that 
\begin{align}
\|  \jb{F_N(u), e_n}_{L^2_x}  \|_{L^2(\rho)}, \ \| \jb{F(u), e_n}_{L^2_x} \|_{L^2(\rho)}
\leq C_k \frac{1}{|n|}
\label{W5a}
\end{align}

\noi
for any $n \in \Z^*$ and any $N \in \N$.	
Moreover, given positive $\eps < \frac 12$, there exists $C_{k, \eps} > 0$ such that 
\begin{align}
\|\jb{ F_M(u) - F_N(u), e_n}_{L^2_x} \|_{L^2(\rho)}
\leq C_{k, \eps} \frac{1}{N^\eps |n|^{1-\eps}}
\label{W5b}
\end{align}

\noi
for any $n \in \Z^*$ and any $ 1 \leq N \leq M \leq \infty$.

\end{lemma}

\begin{proof}
We first prove\;\eqref{W5a}.
Let $N \in \N \cup \{\infty\}$.
Given $x \in \T$, 
it follows from\;\eqref{G5},\;\eqref{W0}, and\;\eqref{W3}
that 
\begin{align}
\P_N u (x) = \s_N^\frac{1}{2}\frac{u_N(x)}{\s_N^\frac{1}{2}}
= \s_N^{\frac 12} \cj{W_{\eta_N(x)}}
= \s_N^{\frac 12} W_{\eta_N(x)}.
\label{W6}
\end{align}

\noi	
Then, from\;\eqref{H2} and\;\eqref{W6}, we have 
\begin{align}
[\P_N u (x)]^k = \s_N^\frac{k}{2} 
\sum_{m = 0}^{[\frac{k}{2}]}
\begin{pmatrix}
k\\2m 
\end{pmatrix}
(2m-1)!! \,  H_{k-2m}(W_{\eta_N(x)}).
\label{W7}
\end{align}

\noi
Clearly, $ \jb{F_N   (u),  e_n}_{L^2_x} = 0$
when $|n|> N$.
Thus, we only need to consider the case $|n|\leq N$.
From Lemma\;\ref{LEM:W1} with\;\eqref{W7},\;\eqref{W3b},  and\;\eqref{W4}, 
we have 
\begin{align}
\|  \jb{F_N  & (u),  e_n}_{L^2_x}  \|_{L^2(\rho)}^2
 = 
 \s_N^k
\int_{\T_x\times\T_y}
e_n(x) \cj{e_n(y)} 
\notag \\
& \hphantom{XX}
\times \sum_{m, \wt m = 0}^{[\frac{k}{2}]}
\begin{pmatrix}
k\\2m 
\end{pmatrix}
(2m-1)!! 
\begin{pmatrix}
k\\2\wt m 
\end{pmatrix}
(2\wt m-1)!!\notag \\
& \hphantom{XX}
 \times \int_\O 
 H_{k-2m}(W_{\eta_N(x)})
  H_{k-2\wt m}(W_{\eta_N(y)}) dPdx dy \notag \\
 & =  
 \sum_{m= 0}^{[\frac{k}{2}]}
\begin{pmatrix}
k\\2m 
\end{pmatrix}^2
\big[(2m-1)!! \big]^2
(k - 2m )! \, 
 \s_N^{2m} 
 \int_{\T_x\times\T_y}
 \big[\g_N(y-x)\big]^{k - 2m}
\cj{ e_n(y-x)}
 dx dy \notag \\
&  = 
 \sum_{m= 0}^{[\frac{k}{2}] }
c_{k, m}\s_N^{2m}
\F\big[\g_N^{k - 2m}\big] (n).
\label{W7a}
\end{align}

Given 
$n = n_1 + \cdots + n_{k - 2m}$, 
we have $\max_j |n_j| \ges |n|$
and thus 
\begin{align}
\F\big[\g_N^{k - 2m}\big] (n)
= \sum_{\substack{n = n_1 + \cdots + n_{k - 2m}\\
1\leq |n_j| \leq N}} 
\prod_{j = 1}^{k-2m} \frac{1}{n_j^2}
\leq d_{k, m}  \frac{1}{n^2}.
\label{W7b}
\end{align}

\noi
Hence,\;\eqref{W5a} follows from\;\eqref{W7a}
and\;\eqref{W7b}.

Next, we prove\;\eqref{W5b}.
Proceeding as before with\;\eqref{W7}, Lemma\;\ref{LEM:W1} for $1\leq N\leq M$, and\;\eqref{W4}, 
we have
\begin{align}
\|\jb{ & F_M(u)  - F_N(u), e_n}_{L^2_x} \|_{L^2(\rho)}^2
 = 
\int_{\T_x\times\T_y}
e_n(x) \cj{e_n(y)} 
\notag \\
& \hphantom{XX}
\times \sum_{m, \wt m = 0}^{[\frac{k}{2}]}
\begin{pmatrix}
k\\2m 
\end{pmatrix}
(2m-1)!! 
\begin{pmatrix}
k\\2\wt m 
\end{pmatrix}
(2\wt m-1)!!\notag \\
& \hphantom{XX}
 \times \int_\O 
\Big[
\ind_{[1, M]}(|n|)
\s_M^k
 H_{k-2m}(W_{\eta_M(x)})
  H_{k-2\wt m}(W_{\eta_M(y)}) \notag \\
& \hphantom{XXXXXX}
-\ind_{[1, N]}(|n|)
 \s_M^\frac{k}{2} \s_N^\frac{k}{2}
  H_{k-2m}(W_{\eta_M(x)})
  H_{k-2\wt m}(W_{\eta_N(y)}) \notag \\
& \hphantom{XXXXXX} 
-\ind_{[1, N]}(|n|)
 \s_M^\frac{k}{2} \s_N^\frac{k}{2}
 H_{k-2m}(W_{\eta_N(x)})
  H_{k-2\wt m}(W_{\eta_M(y)}) \notag \\
&  \hphantom{XXXXXX}
+ \ind_{[1, N]}(|n|)
\s_N^k
 H_{k-2m}(W_{\eta_N(x)})
  H_{k-2\wt m}(W_{\eta_N(y)}) \Big]
    dPdx dy \notag \\
&  = 
 \sum_{m= 0}^{[\frac{k}{2}] }
c_{k, m}
\ind_{[1, N]}(|n|)
\bigg\{
\s_M^{2m}
\F\big[\g_M^{k - 2m}\big] (n)
- 2 
\s_M^{m}\s_N^{m}
\F\big[\g_N^{k - 2m}\big] (n)
+ \s_N^{2m}
\F\big[\g_N^{k - 2m}\big] (n)\bigg\}\notag \\
& 
\hphantom{X}
+  \sum_{m= 0}^{[\frac{k}{2}] }
c_{k, m}
\ind_{(N, M]}(|n|)
\s_M^{2m}
\F\big[\g_M^{k - 2m}\big] (n).
\label{W8a}
\end{align}

On the one hand, noting that $|n| > N$, 
we can use\;\eqref{W7b}
to estimate
the second sum on the right-hand side of\;\eqref{W8a}, 
yielding\;\eqref{W5b}.
On the other hand, noting that 
\begin{align}
\Big|
\F\big[\g_M^{k - 2m}\big] (n)
&  - 
\F\big[\g_N^{k - 2m}\big] (n)\Big| 
 \leq \sum_{\substack{n = n_1 + \cdots + n_{k - 2m}\\
1\leq |n_j| \leq M\\
\max_j|n_j| \geq N}} 
\prod_{j = 1}^{k-2m} \frac{1}{n_j^2}
\leq d_{k, m}  \frac{1}{\max(N,  |n|)^2}
\label{W8b}
\end{align}

\noi
and 
\begin{align}
 |\s_M^m - \s_N^m| \leq C_m |\s_M - \s_N| \les \frac{1}{N},
\label{W8c}
 \end{align}

\noi
we can use 
\eqref{W7b},\;\eqref{W8b}, and\;\eqref{W8c}
to estimate the first sum on the right-hand side of\;\eqref{W8a},
yielding\;\eqref{W5b}.
\end{proof}

As an immediate corollary to Lemma\;\ref{LEM:W5}, 
we obtain the following estimate on the $H^s$-norm of $F_N(u)$, 
establishing Proposition\;\ref{PROP:main} for $p = 2$.

\begin{corollary}\label{COR:nonlin1}
Let $k \geq 2$ be an integer.
Let $ s < \frac 12$.
Then, there exists $C_{k, s} > 0$ such that 
\begin{align}
  \|F_N(u)  \|_{L^2(\rho; H^s)}, \ \| F(u) \|_{L^2(\rho; H^s)}
\leq C_{k, s} 
\label{W9a}
\end{align}

\noi
for any $N \in \N$.	
Moreover, given $\eps >0 $ with $s + \eps < \frac 12$, 
there exists $C_{k, s, \eps} > 0$ such that 
\begin{align}
\| F_M(u) - F_N(u) \|_{L^2(\rho; H^s)}
\leq  \frac{C_{k, s, \eps}}{N^\eps}
\label{W9b}
\end{align}

\noi
for any $ 1 \leq N \leq M \leq \infty$.

\end{corollary}

\subsection{Wiener chaos estimates}
In this subsection, we 
prove  Proposition\;\ref{PROP:main}
by extending\;\eqref{W9a} and\;\eqref{W9b}
in Corollary\;\ref{COR:nonlin1} to any finite $p \geq 2$.
This is achieved by an application of the Wiener chaos estimate (Lemma\;\ref{LEM:hyp3}).

Fix $d \in \N$.\footnote{Indeed, the discussion presented here also holds for $d = \infty$ 
in the context of abstract Wiener spaces.
For simplicity, however, we only consider  finite values for $d$.} Consider the Hilbert space $H = L^2(\R^d, \mu_d)$
endowed with  the Gaussian measure 
 $d\mu_d
= (2\pi)^{-\frac{d}{2}} \exp(-{|x|^2}/{2})dx$, $x = (x_1, \dots,
x_d)\in \R^d$. 
Let $L := \Dl -x \cdot \nabla$ be 
 the Ornstein-Uhlenbeck operator. 
 Then, we have the following  hypercontractivity of the Ornstein-Uhlenbeck
semigroup $S(t) := e^{tL}$ due to Nelson\;\cite{Nelson}.

\begin{lemma} \label{LEM:hyp1}
Let $p \geq 2$.
Then, 
for every $u \in L^p (\R^d, \mu_d)$
and  $t \geq
\frac{1}{2}\log(p-1)$, we have
\begin{equation}\label{hyp1}
\|S(t) u \|_{L^p(\R^d, \mu_d)}\leq \|u\|_{L^2(\R^d, \mu_d)}.
\end{equation}

\noi
\end{lemma}

\noi
We stress that 
the estimate\;\eqref{hyp1} is  independent of the dimension $d$.

Next, we define a {\it homogeneous Wiener chaos of order
$k$} to be an element of the form $\prod_{j = 1}^d H_{k_j}(x_j)$, 
where $k= k_1 + \cdots + k_d$
and $H_{k_j}$ is the Hermite polynomial of degree $k_j$ defined in\;\eqref{H1}. 
Then,  we have the following Ito-Wiener decomposition:
\[ L^2(\R^d, \mu_d) = \bigoplus_{k = 0}^\infty \mathcal{H}_k, \]

\noi
where $\mathcal{H}_k$
is  the closure of homogeneous Wiener chaoses of order $k$
under $L^2(\R^d, \mu_d)$.
We obtain the following corollary to Lemma\;\ref{LEM:hyp1}
for elements in $\mathcal{H}_k$.

\begin{lemma} \label{LEM:hyp2}
Let $F \in \mathcal{H}_k$.
Then, for $p \geq 2$, we have
\begin{equation} \label{hyp2}
\| F \|_{L^p(\R^d, \mu_d)}\leq (p-1)^\frac{k}{2}
\|F\|_{L^2(\R^d, \mu_d)}.
\end{equation}
\end{lemma}

\noi 
It is known that 
any element in $\mathcal H_k$ 
is an eigenfunction of $L$ with eigenvalue $-k$.
Then, the estimate\;\eqref{hyp2} follows immediately
from noting that 
 $F$ is an eigenfunction
of $S(t) = e^{tL}$ with eigenvalue $e^{-tk }$
and setting  $t =\frac{1}{2} \log (p - 1)$ in\;\eqref{hyp1}.

As a further consequence to Lemma\;\ref{LEM:hyp2}, we obtain the following 
Wiener chaos estimate.

\begin{lemma}\label{LEM:hyp3}
Fix $k \in \mathbb{N}$ and $c(n_1, \dots, n_k) \in \C$.
Given 	
 $d \in \mathbb{N}$, 
 let $\{ g_n\}_{n = 1}^d$ be 
 a sequence of  independent standard complex-valued Gaussian random variables
 and set $g_{-n} = \cj{g_n}$.
Define\;$S_k(\o)$ by 
\begin{align*}
 S_k(\o) = \sum_{\G(k, d)} c(n_1, \dots, n_k) g_{n_1} (\o)\cdots g_{n_k}(\o),
 \end{align*}

\noi
where $\G(k, d)$ is defined by
\[ \G(k, d) = \big\{ (n_1, \dots, n_k) \in \{\pm1, \dots,\pm d\}^k \big\}.\]

\noi
Then, for $p \geq 2$, we have
\begin{equation}
 \|S_k \|_{L^p(\O)} \leq \sqrt{k+1}(p-1)^\frac{k}{2}\|S_k\|_{L^2(\O)}.
\label{hyp4}
 \end{equation}

\end{lemma}

Note that the estimate\;\eqref{hyp4} is independent of $d \in \N$.
Lemma\;\ref{LEM:hyp3} follows from\;\eqref{H2}
and Lemma\;\ref{LEM:hyp2}.
See Proposition 2.4 in\;\cite{TTz} for details.
Lemmas\;\ref{LEM:hyp2} and\;\ref{LEM:hyp3}
have been very effective
in the probabilistic study of dispersive PDEs
and related areas.
\cite{Tzv, TTz, Benyi, CO, NORS, R, BTT1}.

We are now ready to present the proof of Proposition\;\ref{PROP:main}.

\begin{proof}[Proof of Proposition\;\ref{PROP:main}]
We only prove\;\eqref{X1} for $N = \infty$.
The proofs of\;\eqref{X1} for $N \in \mathbb{N}$
and\;\eqref{X2} are analogous
in view of Lemma\;\ref{LEM:W5}.

Let $p \geq 2$.
By Minkowski's integral inequality with
\eqref{G5} and\;\eqref{nonlin1}
followed by Lemma\;\ref{LEM:hyp3}
and Lemma\;\ref{LEM:W5}, we have 
\begin{align*}
\|F(u)\|_{L^p(\rho; H^s)} 
& = \Bigg(\sum_{n \in \Z^*} |n|^{2s}
\bigg\|\sum_{\substack{n = n_1+\cdots n_k\\n_j \in \Z^*}}
\prod_{j = 1}^k \frac{g_{n_j}(\o)}{|n_j|}
\bigg\|_{L^p(\O)}^2\Bigg)^\frac{1}{2}\\
& \leq  
\sqrt{k+1}(p-1)^\frac{k}{2}
 \Bigg(\sum_{n \in \Z^*} |n|^{2s}
\bigg\|\sum_{\substack{n = n_1+\cdots n_k\\n_j \in \Z^*}}
\prod_{j = 1}^k \frac{g_{n_j}(\o)}{|n_j|}
\bigg\|_{L^2(\O)}^2\Bigg)^\frac{1}{2}\\
& = 
\sqrt{k+1}(p-1)^\frac{k}{2}
 \bigg(\sum_{n \in \Z^*} |n|^{2s}
\|\jb{F(u), e_n}_{L^2_x}\|_{L^2(\rho)}^2\bigg)^\frac{1}{2}\\
& \leq 
C_k(p-1)^\frac{k}{2}
 \bigg(\sum_{n \in \Z^*} |n|^{2s-2}\bigg)^\frac{1}{2}
\leq C_{k, s} (p-1)^\frac{k}{2}
\end{align*}
	
\noi
as long as  $s < \frac{1}{2}$.
This proves\;\eqref{X1} for $N = \infty$.
\end{proof}

\section{Proof of Theorem~\ref{THM:1}}\label{SEC:proof}

In this section, we present the proof of Theorem~\ref{THM:1}.
Fix an   integer $ k \geq 2$ 
and $s < \frac 12 $
in the following.
The basic structure of the argument follows that in\;\cite{BTT1, OT1}.
We point out a (minor) difference in the presentations in\;\cite{BTT1} and\;\cite{OT1}.
On the one hand, 
the argument 
in\;\cite{BTT1}
was first carried out on a finite time interval
$[-T, T]$ for $T>0$.
Namely, given $T>0$, 
we construct 
a set $\Si_T$ of full probability, guaranteeing the existence
of solutions on $[-T, T]$,
such that 
the law of the random function $u(t)$, $t \in [-T, T]$, is given by~$\Pk$.
Then, the desired set $\Si$ of full probability of global existence
was constructed
as $\Si = \bigcap_{N \in \N}  \Si_N$.
On the other hand, 
the desired set $\Si$ of full probability of global existence
was directly constructed in\;\cite{OT1} without restricting
the argument onto finite time intervals.
In the following, we follow the approach presented in\;\cite{OT1}.

Given $N \in \N$, 
let $\mu_{k, N}$ be the invariant truncated Gibbs measure 
for the truncated gKdV\;\eqref{gKdV2}
constructed in Section\;\ref{SEC:2}.
We first extend $\mu_{k, N}$ to a measure on space-time functions.
Let $\Phi_N: H^s (\T)\to C(\R; H^s(\T))  $
be the solution map to~\eqref{gKdV2} constructed in Lemma~\ref{LEM:global}.
By endowing $ C(\R; H^s(\T))$ with the compact-open topology,\footnote{Under the compact-open topology, a sequence $\{u_n\}_{n \in \N}\subset  C(\R; H^s(\T))$ converges
if and only if it converges uniformly on any compact time interval.}
 it follows from  the local Lipschitz continuity of $\Phi_N(\cdot)$
that $\Phi_N$ is continuous
from $H^s (\T)$ into $C(\R; H^s(\T)) $.
We now 
define a  probability measure $\nu_N$
on $C(\R; H^s(\T))  $ by setting
\begin{align}
 \nu_N = \Pkn \circ \Phi_N^{-1}.
\label{nu}
 \end{align}

\noi
Namely, we define $\nu_N$ as the induced probability measure of $\Pkn$
under the map $\Phi_N$.
In particular, we have
\begin{align*}
 \int_{C(\R; H^s)}  F(u) d\nu_N (u) = \int_{H^s} F(\Phi_N(\phi)) d \Pkn(\phi)
 \end{align*}

\noi
for any measurable function $F :C(\R; H^s(\T))\to  \R$.

Our first goal is to show that 
$\{\nu_N\}_{N\in \N}$ converges to some probability measure\;$\nu$
on\;$C(\R; H^s(\T))$.
For this purpose, 
recall the following definition of tightness for a sequence of probability measures.

\begin{definition}\label{DEF:tight}\rm
A sequence $\{ \rho_n\}_{n \in \N}$ of probability measures
on a metric space\;$\mathcal{S}$ is said to be {\it tight}
if, for every $\eps > 0$, there exists a compact set $K_\eps$
such that $\rho_n(K_\eps^c) \leq \eps$ for all\;${n \in \N}$.
\end{definition}

\noi	
Recall the following 
Prokhorov's theorem on 
a tight  sequence
of probability measures.
See~\cite{Bass}.

\begin{lemma}[Prokhorov's theorem]\label{LEM:Pro}
If a sequence of probability measures 
on a metric space\;$\mathcal{S}$ is tight, then
there is a subsequence that converges weakly to 
a probability measure on\;$\mathcal{S}$.
\end{lemma}

The following proposition shows that 
 the family 
 $\{ \nu_N\}_{N \in \N}$ is tight
 and 
 hence has a subsequence that 
 converges weakly to some probability measure $\nu$ on $C(\R; H^s(\T))$.

\begin{proposition}\label{PROP:tight}
The family $\{ \nu_N\}_{N \in \N}$
of the probability measures on 
$C(\R; H^s(\T))$
is tight.
\end{proposition}

Similar tightness results were proven in\;\cite{BTT1, OT1}
in the context of 
the Gibbs measures for other evolution equations.
Before proceeding to the proof of Proposition\;\ref{PROP:tight}, 
we first state several lemmas.
We use the following notations.
Given $T>0$, we write $L^p_T H^s$ for $L^p([-T, T]; H^s(\T))$.
We use a similar abbreviation for other function spaces in time.

The first lemma provides a uniform control on the size
of random space-time functions under $\nu_N$.
It follows as a consequence of 
the invariance of $\Pkn$ under the dynamics of the truncated gKdV~\eqref{gKdV2}
(Lemma\;\ref{LEM:global}).
See\;\cite{BTT1, OT1} for the proof.

\begin{lemma}\label{LEM:bound1}
Let $ s< \frac 12$ and $p \geq 1$.
Then, there exists $C_p > 0$ such that 
\begin{align*}
\big\| \| u\|_{L^p_T H^s} \big\|_{L^p(\nu_N)} & \leq C_p T^\frac{1}{p}, \\
\big\| \| u\|_{W^{1, p}_T H^{s-3}} \big\|_{L^p(\nu_N)} & \leq C_pT^\frac{1}{p}, 
\end{align*}

\noi
uniformly in $N \in \N$.

\end{lemma}

Recall also the following lemma 
on deterministic functions from~\cite{BTT1}.	

\begin{lemma}[{\cite[Lemma 3.3]{BTT1}}]\label{LEM:BTT1}
Let $T > 0$ and $1\leq p \leq \infty$.
Suppose that 
$u \in L^p_T H^{s_1}$ and $\dt u \in L^p_T H^{s_2}$
for some $s_2 \leq s_1$.
Then, for $ \dl > p^{-1}(s_1 - s_2)$, we have 
\[ \| u \|_{L^\infty_TH^{s_1 - \dl}} \les \| u \|_{L^p_T H^{s_1}}^{1-\frac 1p}
\| u \|_{W^{1, p}_T H^{s_2}}^{ \frac 1p}.\]

\noi
Moreover, there exist $\al > 0$ and $\theta \in [0, 1]$
such that for all $t_1, t_2 \in [-T, T]$, we have
\[ \| u(t_2) - u(t_1)  \|_{H^{s_1 - 2\dl}} \les |t_2 - t_1|^\al  \| u \|_{L^p_T H^{s_1}}^{1-\theta}
\| u \|_{W^{1, p}_T H^{s_2}}^{ \theta}.\]

\end{lemma}
	
We  now present the proof of Proposition~\ref{PROP:tight}.

\begin{proof}[Proof of Proposition~\ref{PROP:tight}]
Let $s < s_1 < s_2 < \frac 12 $.
For $\al \in (0, 1)$ and $T>0$, 
we define  the Lipschitz space $C^\al_TH^{s_1} = C^\al([-T, T]; H^{s_1}(\T))$  by the norm
\[ \| u \|_{C^\al_T H^{s_1}} = \sup_{\substack{t_1, t_2 \in [-T, T]\\t_1 \ne t_2}}
\frac{\| u(t_1) - u(t_2) \|_{H^{s_1}}}{|t_1 - t_2|^\al} + \|u \|_{L^\infty_T H^{s_1}}.
\]

\noi
Note 
that the embedding 
$C^\al_T H^{s_1} \subset 
C_T H^{s}$ is compact
for each $T>0$.
This follows from the compact embedding of $H^{s_1}(\T)$ into $H^s(\T)$
and the H\"older regularity in time
of functions in $C^\al_T H^{s_1}$, 
allowing us to apply Arzel\`a-Ascoli's  theorem.

For $j \in \N$, let $T_j = 2^j$.
Given $\eps > 0$, define $K_\eps$ by 
\[ K_\eps = \big\{ u \in C(\R; H^s):\, \| u \|_{C^\al_{T_j} H^{s_1}} \leq c_0 \eps^{-1} T_j^{1+ \frac{1}{p}} 
\text{ for all }
j \in \N \big\}\]

\noi
for some $p \geq 1$ (to be chosen later).
Let $\{u_n \}_{n \in \N} \subset K_\eps$.
By the definition of $K_\eps$, 
$\{u_n \}_{n \in \N}$ is bounded in 
$C^\al_{T_j} H^{s_1}$ for each $j \in \N$.
Then, in view of the compact embedding 
$C^\al_T H^{s_1} \subset  C_T H^{s}$, 
we can apply the diagonal argument 
to extract a subsequence $\{u_{n_\l} \}_{\l \in \N}$
convergent in $C^\al_{T_j} H^s$ for each $j \in \N$.
In particular,  $\{u_{n_\l} \}_{\l \in \N}$ converges uniformly in $H^s$
on any compact time interval.
Hence,  $\{u_{n_\l} \}_{\l \in \N}$ converges in $C(\R; H^s)$
endowed with the compact-open topology.
This proves that $K_\eps$ is compact in $C(\R; H^s)$.

By Lemma~\ref{LEM:BTT1} with large $p\gg1$ and Young's inequality
followed by Lemma~\ref{LEM:bound1}, 
 we have
\begin{align}
\Big\|\| u \|_{C^\al_T H^{s_1}}\Big\|_{L^p(\nu_N)}
& \les \Big\| \|u \|_{L^p_TH^{s_2}}^{1-\theta} \|u \|_{W^{1, p}_TH^{s_2 - 3}}^{\theta}\Big\|_{L^p(\nu_N)}\notag\\
& \les \Big\|\|u \|_{L^p_TH^{s_2}}\Big\|_{L^p(\nu_N)}
+ \Big\|\|u \|_{W^{1, p}_TH^{s_2 - 3}}\Big\|_{L^p(\nu_N)}
\leq C_p T^\frac{1}{p}.
\label{Y6}
\end{align}
	
\noi
for some $\al \in (0, 1)$ and $\theta \in [0, 1]$, 
uniformly in $N \in \mathbb{N}$.
Then, by Markov's inequality with~\eqref{Y6}
and choosing $c_0 > 0$ sufficiently large, we have 
\[ \nu_N(K_\eps^c) 
\leq c^{-1}_0 \eps T_j^{-1 - \frac{1}{p}}
\Big\|\| u \|_{C^\al_{T_j} H^{s_1}}\Big\|_{L^1(\nu_N)}
\leq c_0^{-1} C_p \eps 
 \sum_{j = 1}^\infty T_j^{-1}
= c^{-1}_0 C_p \eps  < \eps.\]

\noi
This completes the proof of Proposition\;\ref{PROP:tight}.
\end{proof}

As a consequence of  Proposition~\ref{PROP:tight} and Lemma~\ref{LEM:Pro}, 
we conclude that,  
passing to a subsequence,  
$\nu_{N_j}$ converges weakly to some probability measure $\nu $
on $C(\R; H^s(\T))$.

Next, recall 
the following Skorokhod's theorem.
See~\cite{Bass, Dudley} for the proof.

\begin{lemma}[Skorokhod's theorem]\label{LEM:Sk}
Let $\mathcal{S}$ be a 
separable metric space. 
Suppose that $\rho_n$ are probability measures on $\mathcal{S}$
converging weakly to a probability measure $\rho$.
Then, there exist random variables $X_n:\wt \O \to \mathcal{S}$
with laws $\rho_n$
and a random variable  $X:\wt \O \to \mathcal{S}$ with law $\rho$
such that $X_n \to X$ almost surely.

\end{lemma}

It follows from the weak convergence of $\nu_{N_j}$ to $\nu$ and Lemma~\ref{LEM:Sk}
that there exist another probability space $(\wt \O, \wt \F, \wt P)$,
a sequence $\big\{ \wt {u^{N_j}}\big\}_{j \in \N}$ of $C(\R; H^s)$-valued random variables, 
and 
a $C(\R; H^s)$-valued random variable $u $
such that 
\begin{align}
\L\big(\wt {u^{N_j}}\big) = \L( u^{N_j}) = \nu_N, 
\qquad \L(u) = \nu, \label{Y8}
\end{align}
	
\noi
and $\wt {u^{N_j}}$ converges to $u$ in $C(\R; H^s(\T))$ almost surely
with respect to $\wt P$.
Then, Theorem~\ref{THM:1} follows from 
the following proposition.

\begin{proposition}\label{PROP:end}
Let $\wt u_{N_j}$, $j \in \N$,  and $u$ be as above.
Then, 
$\wt {u^{N_j}}$   and $u$
are global-in-time distributional 
solutions to 
the truncated gKdV~\eqref{gKdV2} 
and
to gKdV~\eqref{gKdV1}, respectively.
Moreover,  we have
\begin{align}
\L\big(\wt {u^{N_j}}(t)\big)  = \mu_{k, N_j}
\quad \text{and} \quad
\L(u(t)) = \Pk 
\label{inv10}
\end{align}

\noi
for any $t \in \R$.

\end{proposition}

\begin{proof}

We first prove\;\eqref{inv10}.
Fix $t \in \R$.
Let $R_t:C(\R; H^s) \to H^s$ be the evaluation map defined by $R_t(v) = v(t)$. 
Then, from
Lemma~\ref{LEM:global}, 
we have
\begin{align}
\mu_{k, N_j} = \nu_{N_j}\circ R_t^{-1}.
\label{Y10}
\end{align}

\noi
Denoting by $\nu_{N_j}^t$  the distribution of  $\wt {u^{N_j}}(t)$, 
it follows from~\eqref{Y8} and~\eqref{Y10} that 
\begin{align}
\nu_{N_j}^t = \nu_{N_j}\circ R_t^{-1} = \mu_{k, N_j} .
\label{Y11}
\end{align}

In view of  the almost sure convergence of $\wt {u^{N_j}}$  to $u$ in $C(\R; H^s)$, 
$\wt {u^{N_j}}(t)$ converges to $u(t)$ in $H^s$ almost surely
for any $t \in \R$.
Then, 
denoting by $\nu^t$  the distribution of  $u(t)$, 
it follows from 
 the dominated convergence theorem with~\eqref{Y11} 
 and\;\eqref{inv9} that 
\begin{align*}
\nu^t(A) = \int \ind_{\{u(t)(\o) \in A\}} d\wt P
= \lim_{j\to \infty} \int \ind_{\big\{\wt {u^{N_j}} (t)(\o) \in A\big\}} d\wt P
= \lim_{j \to \infty} \mu_{k, N_j} (A)
= \mu_k(A)
\end{align*}

\noi
for any $A \in \mathcal{B}_{H^s(\T)}$, $s < \frac 12$.
This proves that\;\eqref{inv10}.

Hence, it remains to show that 
$\wt {u^{N_j}}$ and $u$
are global-in-time distributional 
solutions to 
the truncated gKdV~\eqref{gKdV2}
and
to gKdV~\eqref{gKdV1}, respectively.
For $j \in \N$, define the 
$\mathcal D'_{t, x}$-valued
random variable $X_j$ by 
\begin{align*}
X_j =  \dt u^{N_j} + \dx^3  u^{N_j} - \dx \P_{N_j}\big[(\P_{N_j}u^{N_j})^k\big].
\end{align*}

\noi
Here, 
$\mathcal D'_{t, x}= \mathcal{D}'(\R\times \T)$ denotes the space of space-time distributions
on $\R\times \T$.
We define\;$\wt X_j$ 
for $\wt {u^{N_j}}$ in an analogous manner.
Noting that $u^{N_j}$ is a global  solution to~\eqref{gKdV2}, 
we see that $\L_{\mathcal D'_{t, x}}(X_j) = \dl_0$,
where $\dl_0$ denotes the Dirac delta measure.
By~\eqref{Y8}, we also have 
\[\L_{\mathcal D'_{t, x}}(\wt X_j) = \dl_0,\]

\noi
for each $j \in \N$.
In particular, 
 $\wt {u^{N_j}}$ is a global solution
to the truncated gKdV~\eqref{gKdV2}
in the distributional sense, 
almost surely with respect to $\wt P$.

Recall that 
  $\wt {u^{N_j}}$ converges to $u$ in $C(\R; H^s)$ almost surely
  with respect to $\wt P$.
Hence, we have the almost sure convergence
of the linear part:
\begin{align*}
 \dt \wt{u^{N_j}} + \dx^3  \wt {u^{N_j}} 
\ 
\longrightarrow 
\  \dt u + \dx^3 u 
\end{align*}

\noi
in $\mathcal D'(\R\times \T)$ as $j \to \infty$.

Next, we discuss the almost sure convergence of 
the truncated  nonlinearity in the distributional sense.
It suffices to  show
that 
$F_{N_j}\big(\wt{u^{N_j}}\big)= \P_{N_j}\big[(\P_{N_j}u^{N_j})^k\big]$
 converges to $F(u) = u^k$
 in the distributional sense, almost surely.
For simplicity of notation, let $F_j = F_{N_j}$
and $u_j =  \wt{u^{N_j}}$.

Fix $T>0$ and let $s < \frac 12 $.
By Lemma\;\ref{LEM:global} and Proposition~\ref{PROP:main} with\;\eqref{TG1} and\;\eqref{inv8}, 
we have
\begin{align}
\big\|\|F_j (u_j) - F(u_j)\|_{L^2_T H^s}\big\|_{L^2(\nu_{N_j})}
& = 
\big\|\|F_j (\Phi_{N_j} (t) \phi ) - F(\Phi_{N_j}(t) \phi)\|_{L^2(\mu_{k, N_j}) H^s}\big\|_{L^2_T} \notag \\
& = (2T)^\frac{1}{2}
\|F_j( \phi ) - F( \phi)\|_{L^2(\mu_{k, N_j}) H^s} \notag \\
& \les
T^\frac{1}{2} \|R_{N_j}\|_{L^4(\rho)}
\|F_j ( \phi ) - F( \phi)\|_{L^4(\rho) H^s} \notag \\
& \les T^{\frac{1}{2}} N_j^{-\eps},
\label{Y14}
\end{align}

\noi
for some small $\eps > 0$.
In the third step, we used the fact that 
$Z_N \ges 1$
in view of $Z_N = \| R_N(u) \|_{L^1(\rho)}\ \to 
 \| R_\infty(u) \|_{L^1(\rho)}>0$ as $N \to \infty$.
Fix $M \in \N$.
Then, proceeding as in\;\eqref{Y14}, 
we have
\begin{align*}
\big\|\|F (u_j) - F_M(u_j)\|_{L^2_T H^s}\big\|_{L^2(\nu_{N_j})}
& \les T^{\frac{1}{2}} M^{-\eps},
\end{align*}

\noi
uniformly in  $j \in \N \cup \{\infty\}$.
Lastly, note that it follows from 
 the almost sure convergence
of  $\wt {u^{N_j}}$  to $u$ in $C(\R; H^s)$
and the continuity of $F_M$
that
$ F_M (u_j)$ converges to  $F_M (u)$
in\;$C(\R; H^s)$ as $j \to \infty$, almost surely with respect to $\wt P$.
Hence, by 
writing $F_j (u_j) - F(u)$ as 
\begin{align*}
F_j (u_j) - F(u) 
& = \big( F_j (u_j) - F(u_j) \big)
+  \big( F (u_j) - F_M(u_j) \big)\notag \\
& \hphantom{XXX}
+  \big( F_M (u_j) - F_M(u) \big)
+  \big( F_M (u) - F(u) \big), 
\end{align*}
	
\noi
we see that,  
after passing to a subsequence, 
 $F_j(u_j)$ converges to $F(u)$
in $L^2([-T, T]; H^s(\T))$ almost surely with respect to $\wt P$.


By iteratively applying the above argument on time intervals $[-2^\l, 2^\l]$, $\l \in \N$, 
we construct a sequence  $\{\O_\l\}_{\l \in \N}$ of sets of full probability 
with $\O_{\l+1} \subset \O_{\l} $ 
such that 
a subsequence $F_{j^{(\l+1)}}(u_{j^{(\l+1)}})(\o) $ 
of $F_{j^{(\l)}}(u_{j^{(\l)}})(\o)$ from the previous step
converges to $F(u)(\o)$
in $L^2([-2^\l, 2^\l]; H^s(\T))$ for all $\o \in \O_{\l+1}$.
Then, by a diagonal argument, 
passing to a subsequence, the term\;$F_j(u_j) $ converges to $F(u)$
in $L^2_{t, \text{loc}} H^s_x$ 
almost surely with respect to $\wt P$.
In particular, up to a subsequence, 
$F_j(u_j) $ converges to $F(u)$
in $\mathcal D'(\R\times \T)$
almost surely with respect to\;$\wt P$.
Therefore, $u$ is a global-in-time distributional solution to~\eqref{gKdV1}.
\end{proof}

\begin{remark}\label{REM:unique2}\rm
In the proof of Theorem \ref{THM:1} presented above, 
we first fixed $s < \frac{1}{2}$
and thus 
our solution $u$ depends on the value  of  $s < \frac 12$.
In the following, we briefly describe how to remove this dependence on $s$.

Note that the solution map $\Phi_N$ to~\eqref{gKdV2} constructed in Lemma~\ref{LEM:global} 
is independent of $s\geq 0$.
Then, 
letting $s_n = \frac 12 - \frac 1n$, $n \in \N$, 
we can view $\nu_N$ in \eqref{nu}
as a probability measure 
on 
\[C(\R; H^{\frac{1}{2}-}(\T)) :=
 \bigcap_{s < \frac 12} C(\R; H^{s}(\T)) 
=  \bigcap_{n \in \N} C(\R; H^{s_n}(\T))  
\]

\noi
endowed with the following metric
\begin{align}
 d (u, v) = \sum_{n = 1}^\infty \frac{1}{2^n} \frac{\| u - v\|_{C_t H^{s_n}}}{1+ \| u - v\|_{C_t H^{s_n}}}.
\label{nu2}
 \end{align}

It follows from the proof of Theorem \ref{THM:1} that $\{\nu_N\}_{N\in \N}$
is tight as probability measures on 
$C(\R; H^{s_n}(\T))$ for each $n \in \N$.
Then, by Prokhorov's theorem (Lemma \ref{LEM:Pro})
and a diagonal argument, 
we can extract a subsequence
 $\{\nu_{N_j}\}_{j\in \N}$
weakly convergent
to $\nu$
as probability measures on 
$C(\R; H^{s_n}(\T))$ for each $n \in \N$.
In particular, in view of \eqref{nu2}, 
 $\{\nu_{N_j}\}_{j\in \N}$
 converges weakly
to $\nu$
as  probability measures on 
$C(\R; H^{\frac{1}{2}-}(\T))$.
Then, by Skorokhod's theorem (Lemma \ref{LEM:Sk}),
 there exist another probability space $(\wt \O, \wt \F, \wt P)$,
a sequence $\big\{ \wt {u^{N_j}}\big\}_{j \in \N}$ of $C(\R; H^{\frac{1}{2}-})$-valued random variables, 
and 
a $C(\R; H^{\frac{1}{2}-})$-valued random variable $u $
such that \eqref{Y8} holds
and 
$\wt {u^{N_j}}$ converges to $u$ in $C(\R; H^{\frac{1}{2}-}(\T))$ almost surely
with respect to $\wt P$.
This in turn implies that 
$\wt {u^{N_j}}$  converges almost surely to $u$ in $C(\R; H^{s_n}(\T))$ 
for each $n \in \N$.
Finally, by applying 
Proposition \ref{PROP:end}
with some fixed regularity $s_n$, 
we conclude that $u$ is a global distributional solution to \eqref{gKdV1}
and that 
 Theorem \ref{THM:1} holds
with this particular $u\in C(\R; H^{\frac{1}{2}-}(\T))$
for any $s< \frac 12$.

%
%
%
%
%
%
%
%
%
%

\end{remark}

\begin{ackno}\rm
The authors would like to thank the anonymous referee
for raising a question discussed in 
 Remarks \ref{REM:unique} and \ref{REM:unique2}.
T.O.~was supported by the European Research Council (grant no.~637995 ``ProbDynDispEq'').
L.T.~was supported by the grant ``ANA\'E'' ANR-13-BS01-0010-03. 

\end{ackno}

\end{document}